\documentclass{article}
\usepackage[latin1]{inputenc}
\usepackage{longtable,helvet,courier,fancyhdr}
\usepackage{amsfonts}
\usepackage{amsmath,amssymb,bbm}
\usepackage{multicol}
\usepackage{url}
\usepackage{algorithm}
\usepackage{algorithmic}

\def\av{\mathbf{a}}
\def\B{\mathcal{B}}

\def\H{\mathcal{H}}
\def\I{\mathcal{I}}
\def\J{\mathcal{J}}
\def\kk{\mathbbm{k}}

\def\mf{\mathfrak{m}}
\def\NN{\mathbbm{N}}
\def\P{\mathcal{P}}
\def\Q{\mathcal{Q}}
\def\R{\mathcal{R}}
\def\S{\mathcal{S}}

\def\V{\mathcal{V}}

\def\xv{\mathbf{x}}

\DeclareMathOperator{\ch}{char}

\DeclareMathOperator{\gin}{gin}
\DeclareMathOperator{\lt}{lt}

\DeclareMathOperator{\reg}{reg}

\newtheorem{theorem}{Theorem}[section]
\newtheorem{corollary}[theorem]{Corollary}
\newtheorem{lemma}[theorem]{Lemma}
\newtheorem{proposition}[theorem]{Proposition}
\newtheorem{definition}[theorem]{Definition}
\newtheorem{remark}[theorem]{\bf{Remark}}
\newtheorem{example}[theorem]{\bf{Example}}
\newenvironment{proof}{\rm
  \noindent\textbf{Proof}\ }{\hspace*{\fill}
  \begin{math}\Box\end{math}\medskip}

\begin{document}
\title{Deterministically Computing Reduction Numbers of Polynomial Ideals}
\author{Amir Hashemi$^{a,b}$ and Michael Schweinfurter$^{c}$ and Werner M. Seiler$^{c}$\\
 $^{\rm a}$Isfahan University of Technology, Isfahan, 84156-83111, Iran\\
 $^{\rm b}$School of Mathematics, Institute for Research in Fundamental Sciences (IPM),\\ Tehran, 19395-5746, Iran\\
$^{\rm c}$Universit\"at Kassel, Heinrich-Plett-Stra\ss e 40, 34132 Kassel, Germany
}
\maketitle

\begin{abstract}
  We discuss the problem of determining reduction numbers of a polynomial
  ideal $\I$ in $n$ variables.  We present two algorithms based on parametric
  computations.  The first one determines the absolute reduction number of
  $\I$ and requires computations in a polynomial ring with
  $(n-\dim{\I})\dim{\I}$ parameters and $n-\dim{\I}$ variables.  The second
  one computes via a Gr\"obner system the set of all reduction numbers of the
  ideal $\I$ and thus in particular also its big reduction number.  However,
  it requires computations in a ring with $n\dim{\I}$ parameters and $n$
  variables.
\end{abstract}

\section{Introduction}

One of the fundamental ideas behind Gr\"obner bases is the reduction of
questions about general polynomial ideals to monomial ideals.  In the context
of determining invariants of an ideal like projective dimension or
Castelnuovo-Mumford regularity, it is therefore interesting to know when these
invariants possess the same values for an ideal and its leading ideal.  It is
well-known that in many instances the invariants of the leading ideal provide
an upper bound for those of the polynomial ideal and that in generic position,
i.\,e.\ when the leading ideal is the generic initial ideal, the values even
coincide.

From an algorithmic point of view, it is not easy to work with the generic
initial ideal.  While it is comparatively easy to determine it with
probabilistic method, there exists no simple test to verify that one has
really obtained the generic initial ideal.  However, relaxing the conditions
on the leading ideal somewhat one can introduce generic positions which share
many properties with the generic initial ideal and which are effectively
checkable with deterministic algorithms.  In \cite{wms:quasigen}, the authors
showed that for many purposes it suffices to ensure that the leading ideal is
quasi-stable (i.\,e.\  that the given ideal possesses a Pommaret basis
\cite{wms:comb1,wms:comb2}) in order to achieve that many invariants can be
immediately read off the Pommaret basis.

Our article \cite{wms:quasigen} was mainly concerned with invariants and
concepts related to the minimal free resolution of the given ideal.  In this
work, we study the \emph{reduction number}, an invariant which was introduced
by Northcott and Rees \cite{nr:red} and which intuitively measures the
complexity of computations in the associated factor ring.  It is also related
to some other invariants like the degree, the arithmetic degree and the
Castelnuovo-Mumford regularity (see \cite{bh:rednum,nvt:reduct,wvv:rednumalg}
for more details).  Independently, Conca \cite{ac:red} and Trung
\cite{nvt:conrednum} proved that the reduction number of an ideal is bounded
by the one of its leading ideal (for an arbitrary term order) and Trung
\cite{nvt:reduct} showed that for the generic initial ideal (for the degree
reverse lexicographic order) equality holds.

Trung \cite{nvt:conrednum} also presented an approach to the effective
determination of various reduction numbers.  However, his method is very
expensive.  We will show that it is indeed impossible to design a ``simple''
algorithm for reduction numbers where we mean by ``simple'' an approach based
solely on the analysis of leading terms.  Nevertheless, we will provide two
alternative methods which we believe to be more efficient than the one
presented by Trung.  Our first method is based on directly adding the right
number of sufficiently generic linear forms and yields the absolute reduction
number.  Our second method determines the whole set of possible reduction
numbers (and thus in particular both the absolute and the big reduction
number) using a Gr\"obner system.

Throughout this article, we will use the following notations.
$\P=\kk[x_{1},\dots,x_{n}]$ is an $n$-dimensional polynomial ring over some
infinite field $\kk$ with homogeneous maximal ideal $\mf$.  If not stated
otherwise, the term order will always be the degree reverse lexicographic
order induced by $x_{n}\prec\cdots\prec x_{1}$.  We assume that we are given a
fixed homogeneous ideal $\I\unlhd\P$ of dimension $D$ and write for the
corresponding factor ring $\R=\P/\I$.  A non-singular matrix
$A=(a_{ij})\in\mathrm{GL}(n,\kk)$ induces on $\P$ the linear change of
coordinates $x\mapsto A\cdot x$ transforming the given ideal $\I$ into a new
ideal $A\cdot\I\unlhd\P$.  Finally, given a term $t\in\P$, we denote by $w(t)$
the largest integer $\ell$ such that $x_{\ell}\mid t$.

The article is organised as follows.  The next section collects some known
facts about reduction numbers and generic initial ideals.  Section
\ref{sec:genstab} introduces some novel generalised notions of stability for
monomial ideals.  The following section extends the for us crucial notion of
weak $D$-stability to polynomial ideals and presents a deterministic algorithm
to transform any ideal into weakly $D$-stable position.  After these
preparations, we present in Section \ref{sec:absrednum} an algorithm for
computing the absolute reduction number.  In the final section, we exploit
Gr\"obner systems to compute the set of all possible reduction numbers.

\section{Reduction Numbers and the Generic Initial Ideal}
\label{sec:rednumgin}

We recall some basic facts about reduction numbers.  There exist several
equivalent approaches to defining them; for our purposes the following one is
particularly convenient.  Let $y_{1},\dots,y_{D}\in\P_{1}$ be $D$ linear forms
defining a Noether normalisation of $\R$.  Then the ideal $\J=\I+\langle
y_{1},\dots,y_{D}\rangle$ is called a \emph{minimal reduction} of $\I$ and the
\emph{reduction number} $r_{\J}(\R)$ with respect to $\J$ is the largest
non-vanishing degree in the factor ring $\P/\J$.  We write for the set of all
possible reduction numbers $\mathrm{rSet}(\R)=\{r_{\J}(\R)\mid\J\
\text{minimal reduction of}\ \I\}$.  The \emph{(absolute) reduction number}
$r(\R)$ is the minimal element of $\mathrm{rSet}(\R)$, the \emph{big reduction
  number} $br(\R)$ the maximal one.  As already mentioned above, the former
one appeared first in the work of Northcott and Rees \cite{nr:red}; the latter
one was much later introduced by Vasconcelos \cite{wvv:rednumid}.

While it is easy to construct some minimal reduction $\J$, the obvious key
problem in computing $r(\R)$ consists of identifying a $\J$ with
$r_{\J}(\R)=r(\R)$.  In the sequel, we will use the following three results.
The first one characterises all minimal reductions of a \emph{monomial} ideal
in Noether position.  Any such ideal has a minimal generator of the form
$x_{n-D}^{\alpha}$.  The second result relates for a strongly stable ideal
(which is always in Noether position) $r(\R)$ with the exponent $\alpha$.  The
final result bounds for arbitrary ideals $r(\R)$ by $r(\P/\lt{\I})$.

\begin{lemma}[{\cite[Lemma 5]{bh:rednum}}]\label{lem:minred}
  Let $\I\unlhd\P$ be a monomial ideal such that the variables
  $x_{n-D+1},\dots,x_{n}$ induce a minimal reduction.  Then every minimal
  reduction is induced by linear forms
  \begin{equation}\label{eq:mintrafo}
    y_{i}=x_{n-D+i}+\sum_{j=1}^{n-D}a_{ij}x_{j}\,,\quad a_{ij}\in\kk\,.
  \end{equation}
\end{lemma}

\begin{theorem}[{\cite[Thm.~11]{bh:rednum}}]\label{thm:redstable}
  Let $\I\unlhd\P$ be a strongly stable monomial ideal.  Then $\I$ has a
  minimal generator $x_{n-D}^{\alpha}$ and we have $r(\R)=r_{\J}(\R)=\alpha-1$
  for any minimal reduction $\J$ of $\I$.
\end{theorem}

\begin{theorem}[{\cite[Thm.~1.1]{ac:red}, \cite[Cor.~3.4]{nvt:conrednum}}]
  \label{thm:redlti}
  For any ideal $\I\unlhd\P$ and any term order $\prec$, the inequality
  $r(\R)\leq r(\P/\lt{\I})$ holds.
\end{theorem}

Galligo \cite{gall:weier} proved for a base field $\kk$ of characteristic $0$
that almost any linear coordinate transformation leads to the same leading
ideal, the \emph{generic initial ideal} $\gin{\I}$ (for more information see
\cite{mlg:gin}).  Bayer and Stillman \cite{bs:reverse} extended this result to
positive characteristic.  A for us important result of Trung asserts that for
the generic initial ideal the inequality in Theorem \ref{thm:redlti} becomes
an equality.

\begin{theorem}[Galligo, {\cite{gall:weier}, \cite{bs:reverse}}]
  \label{thm:gal}
  There exists a nonempty Zariski open subset $\mathcal{U}\subseteq
  \mathrm{GL}(n,\kk)$ such that $\lt{(A\cdot\I)}=\lt{(A'\cdot\I)}$ for all
  matrices $A,A'\in \mathcal{U}$.
\end{theorem}

\begin{theorem}[{\cite[Thm.~4.3]{nvt:reduct}}]\label{thm:rngin}
  For the degree reverse lexicographic order, we always find
  $r(\R)=r(\P/\gin{\I})$. 
\end{theorem}

\section{Some Generalised Notions of Stability}
\label{sec:genstab}

Stable and strongly stable ideals form two important classes of monomial
ideals.  We introduce now generalisations of these concepts depending on an
integer $\ell$.  In the context of determining reduction numbers, it will turn
out that the case $\ell=D$ is of particular interest.  Like for the classical
stability notions, it is easy to see that it always suffices, if the defining
property is satisfied by the minimal generators of the ideal.

\begin{definition}\label{def:lstab}
  Let $0\leq\ell<n$ be an integer.  The monomial ideal $\I$ is
  \emph{$\ell$-stable}, if for every term $t\in\I$ with $w(t)\ge n-\ell$ and
  every $i<w(t)$ the term $x_it/x_{w(t)}$ also lies in $\I$.  For a
  \emph{weakly $\ell$-stable} ideal $\I$, the above condition must be
  satisfied only for all $i\leq n-\ell$.  Finally, $\I$ is \emph{strongly
    $\ell$-stable}, if for every term $t\in\I$ with $w(t)\ge n-\ell$, every
  $j\ge n-\ell$ with $x_{j}\mid t$ and every $i<j$ the term $x_it/x_{j}$ also
  lies in $\I$.
\end{definition}

\begin{example}\label{ex:katsura}
  We consider first for $n=6$ the ideal
  \begin{displaymath}
    \begin{array}{ll}
      \I={}&\langle x_{1},\ x_4^2,\ x_3x_4,\ x_{2}x_4,\ x_{2}x_3,\ x_2^2,\ x_5^3,\
      x_4x_5^2,\ x_3x_5^2,\ x_2x_5^2,\ x_3^2x_5,\ x_3^3,\ x_5^2x_6^2,\\ 
      & x_4x_5x_6^2,\ x_3x_5x_6^2,\ x_2x_5x_6^2,\ x_3^2x_6^2,\ x_5x_6^4,\ x_4x_6^4,\
      x_3x_6^4,\ x_2x_6^4,\ x_6^6 \rangle\,,
    \end{array}
  \end{displaymath}
  the leading ideal of the fifth Katsura ideal.  As one can easily see that
  here $D=0$, it suffices to check the defining property for the generators
  containing $x_{6}$ and it turns out that $\I$ is $0$-stable.  However, $\I$
  is not stable, as for example $x_{3}x_{4}\in\I$ but $x_{3}^{2}\notin\I$.

  Consider now for $n=5$ the monomial ideal
  \begin{displaymath}
   \begin{array}{ll}
      \I={}& \langle x_1^2,\ x_2^3,\ x_1x_2^2,\ x_3^2x_2^2,\ x_2x_3^2x_1,\ x_3^5,\
      x_2x_3^4,\ x_1x_3^4,\ x_3^4x_4^2,\ x_2x_3^3x_4^2,\ x_1x_3^3x_4^2,\\
      & x_3^3x_4^4,\ x_3^2x_2x_4^4,\ x_1x_3^2x_4^4,\ x_3x_2^2x_4^4,\ 
      x_3x_2x_1x_4^4,\ x_1x_2x_3x_4^3x_5^2,\ x_1x_3x_4^6,\ x_2^2x_4^6,\\
      & x_1x_2x_4^6,\ x_2^2x_4^5x_5^2,\ x_1x_2x_4^5x_5^2,\ x_3x_2^2x_4^3x_5^4,\
      x_1x_3x_4^5x_5^3,\ x_2x_3^2x_4^3x_5^5,\ x_1x_3^2x_4^3x_5^5,\\
      & x_1x_2x_4^4x_5^6,\ x_1x_4^6x_5^5,\ x_2^2x_4^4x_5^7,\
      x_2x_3x_4^5x_5^7\rangle\,. 
    \end{array}
  \end{displaymath}
  Since here $D=2$, we must check the defining property of a weakly $D$-stable
  ideal only for the terms containing $x_{3},x_{4},x_{5}$ and one readily
  verifies that $\I$ is weakly $D$-stable.  However, it is not $D$-stable
  because $t=x_{1}x_{4}^6x_{5}^5\in\I$ but $tx_{4}/x_{5}\notin\I$.
\end{example}

The generic initial ideal is always Borel-fixed, i.\,e.\ invariant under the
natural action of the Borel group \cite{bs:reverse,ag:divstab}.  In general,
it depends on the characteristic of the base field whether a given ideal is
Borel-fixed.  In characteristic zero, the Borel-fixed ideals are precisely the
strongly stable ones.  We provide now the analogous result for strong
$\ell$-stability.

\begin{definition}\label{def:lborel}
  The \emph{Borel group} is the subgroup $\B<\mathrm{GL}(n,\kk)$ consisting of
  all lower triangular invertible $n\times n$ matrices.  For any integer
  $0\leq\ell<n$, we define the \emph{$\ell$-Borel group} as the subgroup
  $\B_{\ell}\leq\B$ consisting of all matrices $A\in\B$ such that for
  $i<n-\ell$ we have $a_{ii}=1$ and $a_{ij}=0$ for $i\neq j$.
\end{definition}

\begin{proposition}\label{prop:lstabborel}
  Assume that $\ch{\kk}=0$.  The monomial ideal $\I\unlhd\P$ is strongly
  $\ell$-stable, if and only if it is invariant under the $\ell$-Borel group
  $\B_{\ell}$.
\end{proposition}

\begin{proof}
  Assume first that $\I$ is $\ell$-stable and consider a generating set $\H$
  of it.  The transformation induced by an element $A=(a_{ij})\in\B_{\ell}$ is
  of the form
  \begin{equation}\label{eq:lborel}
    \begin{array}{ll}
      x_{i}\rightarrow x_{i} & \mbox{if } i< n-\ell\,, \\
      x_{i}\rightarrow a_{ii}x_{i}+\sum_{j=n-l}^{i-1}a_{ij}x_{j}\quad & 
          \mbox{if } i\geq n-\ell\,. \\
    \end{array}
  \end{equation}
  One immediately sees that any generator $t\in\H$ with $w(t)<n-\ell$ remains
  unchanged under the action of $A$.  If $w(t)\geq n-\ell$, then $t$ is
  transformed into a polynomial $f_{t}=A\cdot t$.  It follows again from
  \eqref{eq:lborel} that any term in the support of $f_{t}$ is obtained from
  $t$ by applying a sequence of ``elementary moves'' of the form $s\rightarrow
  x_{j}s/x_{k}$ with $j<k$ where $x_{k}\mid s$.  In this sequence we always
  have $k\geq n-\ell$ and thus the strong $\ell$-stability of $\I$ implies
  that all appearing terms $s$ lie in $\I$.  Furthermore, $t$ itself always
  lies in the support of $f_{t}$.

  Consider now the elements $t$ of $\H$ with $w(t)\geq n-\ell$ sorted reverse
  lexicographically.  If $t$ is the largest term among these, then $w(s)<w(t)$
  for all $s\neq t$ appearing in the support of $f_{t}$. Thus they are
  multiples of elements of $\H$ which remain unchanged under the operation of
  $A$ and can be eliminated.  If $t$ is the second largest term, then the
  support may in addition contain multiples of the largest term; otherwise we
  can apply the same argument.  By iteration, we obtain that the whole ideal
  remains invariant.

  For the converse, we need the assumption on the characteristic.  If
  $\ch{\kk}=0$ (and thus no coefficient drops out when we transform a term),
  then we may revert the above arguments: if $\I$ is invariant under
  $\B_{\ell}$, then all terms appearing in the support of $f_{t}$ must lie in
  $\I$ and hence $\I$ is strongly $\ell$-stable.
\end{proof}

In relation to our previous work \cite{wms:quasigen}, it is of interest to
show that a $D$-stable ideal is automatically quasi-stable.  The proof depends
on the following characterisation of $\ell$-stability which is of independent
interest.

\begin{proposition}\label{prop:lstab}
  The monomial ideal $\I\unlhd\P$ is $\ell$-stable, if and only if it
  satisfies for all $0\leq i\leq\ell$
  \begin{equation}\label{eq:lstab}
    \langle\I,x_{n},\dots,x_{n-i+1}\rangle:x_{n-i}=
    \langle\I,x_{n},\dots,x_{n-i+1}\rangle:\mf\,.
  \end{equation}
\end{proposition}

\begin{proof}
  Assume first that $\I$ is $\ell$-stable and let $t$ be a term such that
  $x_{n-i}t\in\langle\I,x_{n},\dots,x_{n-i+1}\rangle$ for some $i\leq\ell$.
  If $w(t)>n-i$, then $t\in\langle x_{n},\dots,x_{n-i+1}\rangle$ and nothing
  is to be proven.  Otherwise we have $x_{n-i}t\in\I$ and $w(x_{n-i}t)=n-i\geq
  n-\ell$.  Because of the $\ell$-stability, this entails that
  $x_{j}t=x_{j}(x_{n-i}t)/x_{n-i}\in\I$ for all $j\leq n-\ell$.  Hence
  $t\langle x_{1},\dots,x_{n-i}\rangle\subseteq\I$ implying
  $t\mf\subseteq\langle\I,x_{n},\dots,x_{n-i+1}\rangle$.

  For the converse consider a term $t\in\I$ with $w(t)=n-i\geq n-\ell$.
  Because of (\ref{eq:lstab}), we have
  $t/x_{n-i}\in\I:x_{n-i}\subseteq\langle\I,x_{n},\dots,x_{n-i+1}\rangle:\mf$.
  Hence $x_{j}t/x_{n-i}\in\langle\I,x_{n},\dots,x_{n-i+1}\rangle$ for all
  $j\leq n$.  If $j\leq n-i$, then $w(x_{j}t/x_{n-i})\leq n-i$ and thus we
  must have $x_{j}t/x_{n-i}\in\I$ so that $\I$ is $\ell$-stable.
\end{proof}

\begin{corollary}\label{cor:dstabqstab}
  A $D$-stable monomial ideal $\I$ is quasi-stable.
\end{corollary}

\begin{proof}
  According to the previous proposition, (\ref{eq:lstab}) holds for all $0\leq
  i\leq D$.  As a preparatory step, we claim that this fact implies that for
  these values of $i$ also
  \begin{equation}\label{eq:lstabsat}
    \langle\I,x_{n},\dots,x_{n-i+1}\rangle:x_{n-i}^{\infty}=
    \langle\I,x_{n},\dots,x_{n-i+1}\rangle:\mf^{\infty}\,.
  \end{equation}
  Indeed, if the term $t$ lies in the ideal on the left hand side, then an
  integer $s$ exists such that
  $x_{n-i}^{s}t\in\langle\I,x_{n},\dots,x_{n-i+1}\rangle$ and therefore
  \begin{displaymath}
    x_{n-i}^{s-1}t\in\langle\I,x_{n},\dots,x_{n-i+1}\rangle:x_{n-i}=
    \langle\I,x_{n},\dots,x_{n-i+1}\rangle:\mf\,.
  \end{displaymath}
  Applying this argument a second time yields
$$  \begin{array}{ll}
    x_{n-i}^{s-2}t &\in \bigl(\langle\I,x_{n},\dots,x_{n-i+1}\rangle
                        :\mf\bigr):x_{n-i}\\

    &= \bigl(\langle\I,x_{n},\dots,x_{n-i+1}\rangle:x_{n-i}\bigr):\mf\\
    &= \langle\I,x_{n},\dots,x_{n-i+1}\rangle:\mf^{2}\,.
  \end{array}$$
  Thus we find by iteration that
  $t\in\langle\I,x_{n},\dots,x_{n-i+1}\rangle:\mf^{s}$ proving the claim.

  It follows that $x_{n-i}$ is not a zero divisor in
  $\P/(\langle\I,x_{n},\dots,x_{n-i+1}\rangle:\mf^{\infty})$ for all $0\leq i<
  D$.  Indeed, if $f\in\P$ satisfies
  $x_{n-i}f\in\langle\I,x_{n},\dots,x_{n-i+1}\rangle:\mf^{\infty}$, then an
  exponent $s$ exists such that
  $x_{n-i}f\mf^{s}\subseteq\langle\I,x_{n},\dots,x_{n-i+1}\rangle$ and hence
  $x_{n-i}^{s+1}f\in\langle\I,x_{n},\dots,x_{n-i+1}\rangle$.  But this
  implies $f\in \langle\I,x_{n},\dots,x_{n-i+1}\rangle:x_{n}^{\infty}=
  \langle\I,x_{n},\dots,x_{n-i+1}\rangle:\mf^{\infty}$.  Now the assertion
  follows from \cite[Prop.~4.4]{wms:comb2}.
\end{proof}

\begin{example}
  Weak $D$-stability is not sufficient for quasi-stability, as one can see
  from the ideal $\langle x_{1}^{2},x_{1}x_{3}\rangle$ where $n=3$ and $D=2$.
  One easily verifies that it is weakly $D$-stable but not quasi-stable.  And
  for the same values of $n$ and $D$ the ideal $\langle x_{1}^{3},
  x_{1}x_{2}\rangle$ shows that the converse of Corollary \ref{cor:dstabqstab}
  does not hold, as it is quasi-stable but not (weakly) $D$-stable.
\end{example}

\begin{remark}\label{rem:weaklstab}
  Assume that the monomial ideal $\I$ is weakly $\ell$-stable for some $\ell$
  and that $t=x_{1}^{\alpha_{1}}\cdots x_{n}^{\alpha_{n}}\in\I$.  It follows
  immediately from Definition \ref{def:lstab} that any term of the form
  $x_{1}^{\alpha_{1}+\beta_{1}}\cdots
  x_{n-\ell}^{\alpha_{n-\ell}+\beta_{n-\ell}}$ with
  $\beta_{1}+\cdots+\beta_{n-\ell}= \alpha_{n-\ell+1}+\cdots+\alpha_{n}$ is
  then also contained in $\I$.  If we introduce for $1\leq j\leq\ell$ the
  homogeneous polynomials
  \begin{displaymath}
    g_{j}=\sum_{\beta_{1}^{(j)}+\cdots+\beta_{n-\ell}^{(j)}=\alpha_{n-\ell+j}}
        a^{(j)}_{\beta_{1}^{(j)},\dots,\beta_{n-\ell}^{(j)}}
        x_{1}^{\beta_{1}^{(j)}}\cdots x_{n-\ell}^{\beta_{n-\ell}^{(j)}}
  \end{displaymath}
  with arbitrary coefficients
  $a^{(j)}_{\beta_{1}^{(j)},\dots,\beta_{n-\ell}^{(j)}}\in\kk$, then it
  follows from the observation above that the polynomial
  \begin{displaymath}
    f_{t}=x_{1}^{\alpha_{1}}\cdots x_{n-\ell}^{\alpha_{n-\ell}}g_{1}\cdots g_{\ell}
  \end{displaymath}
  also lies in $\I$.  Each term in its support is of the form
  $x_{1}^{\alpha_{1}+\beta_{1}}\cdots
  x_{n-\ell}^{\alpha_{n-\ell}+\beta_{n-\ell}}$ with
  $\beta_{i}=\beta_{i}^{(1)}+\cdots+\beta_{i}^{(\ell)}$ and by construction
  $\beta_{1}+\cdots+\beta_{n-\ell}= \alpha_{n-\ell+1}+\cdots+\alpha_{n}$.
\end{remark}

\begin{proposition}\label{prop:wdsnp}
  A weakly $D$-stable ideal $\I$ is always in Noether position.
\end{proposition}

\begin{proof}
  A $D$-dimensional monomial ideal is in Noether position, if and only if for
  all $1\leq j\leq n-D$ a pure power $x_{j}^{e_{j}}$ is contained in $\I$.
  Assume first that there exists a term
  $t\in\I\cap\kk[x_{n-D+1},\dots,x_{n}]$.  Then Remark \ref{rem:weaklstab}
  immediately implies for $e=\deg{t}$ that $x_{j}^{e}\in\I$ for all $1\leq
  j\leq n-D$ and we are done.  If
  $\I\cap\kk[x_{n-D+1},\dots,x_{n}]=\emptyset$, then the $D$-dimensional cone
  $1\cdot\kk[x_{n-D+1},\dots,x_{n}]$ lies completely in the complement of
  $\I$.  Assume that for some $1\leq j\leq n-D$ no power of $x_{j}$ was
  contained in $\I$.  Since $D=\dim{\I}$, it is not possible that the
  complement of $\I$ contains a $(D+1)$-dimensional cone.  Thus we must have
  $\I\cap\kk[x_{j},x_{n-D+1},\dots,x_{n}]\neq\emptyset$.  But if a term $t$ of
  degree $e$ lies in this intersection, then again by Remark
  \ref{rem:weaklstab} $x_{j}^{e}\in\I$ in contradiction to our assumption.
\end{proof}

The simple Algorithm \ref{alg:wdstest} verifies whether a given monomial ideal
is weakly $D$-stable without a priori knowledge of the dimension $D$ of $\I$.
For showing its correctness, we note that if $\I$ is weakly $D$-stable, then
the number $d$ computed in Line 2 equals $D$ by Proposition \ref{prop:wdsnp}
and by Definition \ref{def:lstab} of weak $D$-stability we never get to Line
6.  If $\I$ is not weakly $D$-stable, then $d\geq D$ (this estimate holds for
any monomial ideal) and soon or later we will reach Line 6.  The bit
complexity of the algorithm is polynomial in $kn$, as one can easily see that
the number of operations in the two \textbf{for}-loops is at most $k^2n^3$.

\begin{algorithm}
\caption{\textsc{WDS-Test}: Test for weak $D$-stability}\label{alg:wdstest}
\begin{algorithmic}[1]
  \REQUIRE minimal basis $G=\{m_{1},\dots,m_{k}\}$ of monomial ideal $\I\lhd\P$
  \ENSURE The answer to: is $\I$ weakly $D$-stable?
  \STATE $e:=\max{\{\deg(m_1),\ldots ,\deg(m_k)\}}$

  \STATE $d:=$ smallest $\ell$ such that $x_i^{e}\in\I$ for $i=1,\ldots,n-\ell$
  \FORALL {$x_{1}^{e_{1}}\cdots x_{h}^{e_{h}}\in G$ with $h\ge n-d$ and $e_{h}>0$}
    \FOR {$j=1,\ldots ,n-d$} 
      \IF {$x_{1}^{e_{1}}\cdots x_{h-1}^{e_{h-1}}x_{h}^{e_{h}-1}x_{j}\notin 
           \langle G\rangle$}
        \STATE \textbf{return} false
      \ENDIF
    \ENDFOR
  \ENDFOR 
  \STATE \textbf{return} true
\end{algorithmic}
\end{algorithm}

\section{Weak $D$-Stability for Polynomial Ideals}
\label{sec:polywdstab}

In the previous section, we considered exclusively monomial ideals.  All the
notions introduced in Definition \ref{def:lstab} can be straightforwardly
extended to polynomial ideals by saying that an ideal $\I$ satisfies some form
of stability, if its leading ideal $\lt{\I}$ satisfies this form of stability.
Galligo's Theorem \ref{thm:gal} immediately implies that after a generic
change of coordinate $A\in\mathrm{GL}(n,\kk)$ the transformed ideal $A\cdot\I$
possesses any stability property here considered.  Thus in principle a random
coordinate transformation (almost) always provides a ``nice'' leading ideal.

However, from a computational point of view, random transformations are rather
unpleasant, as they destroy all sparsity typically present in ideal bases.  It
is therefore of great interest to see whether for some notion of stability it
is possible to design a \emph{deterministic} algorithm which yields a fairly
sparse transformation $A$ such that $A\cdot\I$ has the desired stability
property.  In a forthcoming work \cite{wms:effgen}, we will study this
question in depth and provide such an algorithm for many important stability
notions.  Here, we only present a variation of this algorithm for the case of
weak $D$-stability.  For lack of space, we omit the (non-trivial) termination
proof which will be given in \cite{wms:effgen}.

Algorithm \ref{algo:wdstrafo} works by performing incrementally very sparse
transformations where all variables except one remain unchanged and this one
undergoes a transformation of the form $x_{i}\rightarrow x_{i}+ax_{j}$ where
$j<i$ and $a\in\kk\setminus\{0\}$ is a generic parameter.  The pair $(i,j)$ is
chosen in such a way that each transformation leads to true progress towards a
weakly $D$-stable position, if $a$ does not take one of finitely many ``bad''
values.  In practice, we always use the value $a=1$.  If this accidentally
represents a ``bad'' value, then we will automatically perform the same
transformation a second time which corresponds to $a=2$.  Obviously, after a
finite number of iterations (which can be bounded via the degrees of the
generators), we will reach a ``good'' value, since $\kk$ is an infinite field.

\begin{algorithm}
  \caption{\textsc{WDS-Trafo}: Transformation to weakly $D$-stable
    position}\label{algo:wdstrafo} 
  \begin{algorithmic}[1]
    \REQUIRE Gr\"obner basis $G$ of homogeneous ideal $\I\unlhd\P$
    \ENSURE  a linear change of coordinates $\Psi$ such that $\Psi(\I)$ is
             weakly $D$-stable 
    \STATE $D:=\dim{\I}$; $\Psi:=\mathrm{id}$
    \WHILE {$\exists\, g\in G,\ 1\leq j\leq n-D\,:\,
                 i=w(\lt{g})\geq n-D\wedge 
                 x_{j}\lt{(g)}/x_{i}\notin\langle\lt{G}\rangle$}
        \STATE $\psi:=(x_{i}\mapsto x_{i}+x_{j})$; $\Psi=\psi\circ\Psi$
        \STATE $G:=\text{\textsc{Gr\"obnerBasis}}\bigl(\psi(G)\bigr)$
    \ENDWHILE
    \STATE \textbf{return} $\Psi$
  \end{algorithmic}
\end{algorithm}

Algorithm \ref{algo:wdstrafo} is not in an optimised form.  In practice, if
one finds more than one suitable pair $(i,j)$, it appears natural to perform
several transformations simultaneously, as each iteration of the
\textbf{while} loop requires a Gr\"obner basis computation.  Furthermore, one
should take into account that the input for these computations is typically
already fairly close to a Gr\"obner basis.  Hence it is probably useful to
apply some specialised algorithm exploiting this fact.  A prototype
implementation of Algorithm \ref{algo:wdstrafo} in \textsc{Maple} can be found
at \url{http://amirhashemi.iut.ac.ir/softwares}.

\begin{example}
  We consider for $n=3$ the ideal $\I=\langle x_{1}^{3}, x_{2}^{2}x_{3},
  x_{2}^{3}\rangle$ with $D=1$.  This ideal is not weakly $D$-stable, since
  $x_{1}(x_{2}^{2}x_{3})/x_{3}\notin\I$ and, according to Algorithm
  \ref{algo:wdstrafo}, we perform the change of coordinates
  $\psi_{1}:x_{3}\mapsto x_{1}+x_{3}$.  The transformed ideal
  $\I_{1}=\psi_{1}(\I)$ has the leading ideal $\langle x_{1}^{3},
  x_{1}x_{2}^{2}, x_{2}^{3}, x_{2}^{2}x_{3}^{3}\rangle$ and is also not
  $D$-stable, since $x_{1}(x_{1}x_{2}^{2})/x_{2}\notin\lt{\I_{1}}$.  Thus in
  the second iteration the \textbf{while} loop performs the change of
  coordinate $\psi_{2}:x_{2}\mapsto x_{1}+x_{2}$.  The leading ideal of the
  transformed ideal $\I_{2}=\psi_{2}(\I_{1})$ is by chance even the generic
  initial ideal $\gin{\I}=\langle x_{1}^{3}, x_{1}^{2}x_{2}, x_{1}x_{2}^{2},
  x_{2}^{4}, x_{1}^{2}x_{3}^{3}\rangle$ and thus of course weakly $D$-stable.
\end{example}

\section{Computing the Absolute Reduction Number}
\label{sec:absrednum}

We consider first the case of a monomial ideal and extend Theorem
\ref{thm:redstable} from strongly stable ideals to weakly $D$-stable ones.
Our proof follows closely the arguments of the original proof by Bresinsky and
Hoa \cite{bh:rednum}.

\begin{theorem}\label{thm:rnwds}
  Let $\I\unlhd\P$ be a weakly $D$-stable monomial ideal.  Then $\I$ has a
  minimal generator $x_{n-D}^{\alpha}$ and $r(\R)=r_{\J}(\R)=\alpha-1$ for any
  minimal reduction~$\J$ of $\I$.
\end{theorem}

\begin{proof}
  Since $\I$ is assumed to be weakly $D$-stable, $x_{n-D+1},\dots,x_{n}$
  induce a minimal reduction by Proposition \ref{prop:wdsnp} and we can apply
  Lemma \ref{lem:minred}.  Consider the $D$ linear forms
  $y_i=x_{n-D+i}+a_{i,1}x_1+\cdots+a_{i,n-D}x_{n-D}$ with $1\leq i\leq D$ and
  arbitrary coefficients $a_{i,j}\in \kk$ and set $\J_{1}=\I+\langle
  y_1,\ldots,y_D \rangle$.

  We claim that $r_{\J_{1}}(\R)=r_{\J_{2}}(\R)$ where $\J_{2}=\I+\langle
  x_{n-D+1},\dots,x_{n} \rangle$.  It is enough to show the identity
  $\I_1=\I_2$ where $\P/\J_{1}\simeq \kk[x_1,\ldots,x_{n-D}]/\I_1$ and
  $\P/\J_{2}\simeq \kk[x_1,\ldots,x_{n-D}]/\I_2$.  One easily sees that
  $\I_2=\I \cap \kk[x_1,\ldots,x_{n-D}]$ and thus trivially $\I_2\subseteq
  \I_1$.  The converse inclusion $\I_1\subseteq \I_2$ follows by Remark
  \ref{rem:weaklstab} which entails that for any term $t=x_1^{\alpha_1}\cdots
  x_n^{\alpha_n}\in \I$ the corresponding term
  \begin{displaymath}
    \tilde t=x_1^{\alpha_1}\cdots x_{n-D}^{\alpha_{n-D}}
    \prod_{j=1}^{D}(-a_{j,1}x_1-\cdots-a_{j,n-D}x_{n-D})^{\alpha_{n-D+j}}\in\I_{1}
  \end{displaymath}
  also lies in $\I$ and hence in $\I_{2}$.

  Proposition \ref{prop:wdsnp} also implies that $\I$ has a minimal generator
  of the form $x_{n-D}^\alpha$ for some $\alpha\in\NN$.  Hence,
  $r_{\J_{2}}(\R)\ge \alpha-1$.  On the other hand, $x_{n-D}^\alpha\in \I$
  implies by Remark \ref{rem:weaklstab} that any term $x_1^{\alpha_1}\cdots
  x_{n-D}^{\alpha_{n-D}}$ of degree $\alpha$ also belongs to $\I$ and thus
  $r_{\J_{2}}(\R)\le \alpha-1$. Therefore $r_{\J_{2}}(\R)=\alpha-1$ proving
  the second assertion. 
\end{proof}

We have thus identified a class of monomial ideals, the weakly $D$-stable
ideals, for which it is particularly simple to determine their reduction
number.  Given a polynomial ideal $\I$, we may use Algorithm
\ref{algo:wdstrafo} to render it weakly $D$-stable and obtain then immediately
the reduction number of its leading ideal $\lt{\I}$.  According to Theorem
\ref{thm:redlti}, this number gives us an upper bound for $r(\R)$.  We
introduce now a more specialised class of ideals for which we can guarantee
that $\I$ and $\lt{\I}$ have the same reduction number.  We denote here for a
monomial ideal ${\mathcal{L}}$ by $\deg_{x_{k}}{{\mathcal{L}}}$ the maximal
$x_{k}$-degree of a minimal generator of ${\mathcal{L}}$.

\begin{definition}\label{weakid}
  Let $0\leq\ell<n$ be an integer.  The homogeneous ideal $\I\unlhd\P$ is
  \emph{weakly $\ell$-minimal stable}, if its leading ideal $\lt{\I}$ is
  weakly $\ell$-stable and if for any linear change of coordinates
  $A\in\mathrm{GL}(n,\kk)$ such that $\lt{(A\cdot\I)}$ is still weakly
  $\ell$-stable, we have $\deg_{x_{n-\ell}}{\lt{\I}}\le
  \deg_{x_{n-\ell}}{\lt{(A\cdot\I)}}$.
\end{definition}

Again it is easy to see that this is a generic notion, as any coordinate
transformation $A$ with $\lt{(A\cdot\I)}=\gin{\I}$ leads to a weakly
$\ell$-minimal stable position.

\begin{example}\label{ex:green}
  Consider for $n=3$ the ideal $\I=\langle x_1x_3, x_1x_2+x_2^2, x_1^2\rangle$
  introduced by Green \cite{mlg:gin}.  One finds that the leading ideal
  $\lt{\I}=\langle x_1^2,x_1x_2,x_1x_3,x_2^3,x_2^2x_3\rangle$ is even strongly
  stable and thus of course weakly $D$-stable (with $D=1$ here).  However,
  $\I$ is not weakly $D$-minimal stable, as $\gin(\I)=\langle x_1^2, x_1x_2,
  x_2^2, x_1x_3^2 \rangle$ and thus has a lower degree in $x_{2}$.
\end{example}

\begin{example}\label{ex:bh15}
  We consider for $n=4$ the ideal
  \begin{displaymath}
    \I=\langle x_1x_4-x_2x_3,\ x_2^3-x_1x_3^2,\ x_2^2x_4-x_1^3\rangle\,;
  \end{displaymath}
  it represents the special case $a=2$, $b=3$ of \cite[Example 15]{bh:rednum}.
  Here $D=2$ and the ideal $\I$ is not weakly $D$-stable.  The following
  linear change of coordinates $\Psi:x_{2}\mapsto x_{1}+x_{2}, x_{3}\mapsto
  x_{1}+x_{3}$ transforms $\I$ into a weakly $D$-stable (in fact, even
  strongly stable) ideal $\I_{1}$ with leading ideal
  \begin{displaymath}
    \lt{\I_{1}}=\langle x_1^2,\ x_1x_2^2,\ x_2^3,\ x_1x_2x_3^2,\ x_1x_3^3,\ 
                        x_2^2x_3^3,\ x_2x_3^4\rangle\,.
  \end{displaymath}
  Note that although this leading ideal is different from
  \begin{displaymath}
    \gin{\I}=\langle x_1^2,\ x_1x_2^2,\ x_2^3,\ x_1x_2x_3^2,\ x_2^2x_3^2,\ 
                     x_1x_3^4,\ x_2x_3^4\rangle\,,
  \end{displaymath}
  both ideals have the same minimal generator $x_{2}^{3}$.  Thus $\I_{1}$ is
  weakly $D$-minimal stable and we see that in this example the set of
  transformations leading to weakly $D$-minimal position is strictly larger
  than the one leading to the generic initial ideal.
\end{example}

\begin{theorem}\label{thm:wdms}
  Let $\I\unlhd\P$ be a weakly $D$-minimal stable homogeneous ideal.  Then
  $\lt{\I}$ has a minimal generator $x_{n-D}^\alpha$ and
  $r(\R)=r(\P/\lt{\I})=\alpha-1$.
\end{theorem}

\begin{proof}
  Since $\lt{\I}$ is weakly $D$-stable, it possesses by Proposition
  \ref{prop:wdsnp} a minimal generator $x_{n-D}^{\alpha}$ and thus
  $r(\P/\lt{\I})=\alpha-1$ by Proposition \ref{thm:rnwds}.  As $\I$ is assumed
  to be weakly $D$-minimal stable, $x_{n-D}^{\alpha}$ must also be a minimal
  generator of $\gin{\I}$ and hence $r(\R)=r(\P/\gin{\I})=\alpha-1$ by Theorem
  \ref{thm:rngin}.
\end{proof}

Unfortunately, Theorem \ref{thm:wdms} is mainly of theoretical interest, as we
are not able to provide a simple deterministic algorithm for the construction
of a change of coordinates leading to be weakly $D$-minimal stable position.
We present now Algorithm \ref{alg:rednum} for the computation of $r(\R)$.
Instead of a coordinate transformation, it is based on a parametric
computation.  The main point will be to keep the number of parameters as small
as possible.

\begin{algorithm}
  \caption{\textsc{RedNum}: (Absolute) Reduction Number}\label{alg:rednum}
  \begin{algorithmic}[1]
    \REQUIRE Gr\"obner basis $G$ of a homogeneous ideal $\I\lhd\P$
    \ENSURE the absolute reduction number $r(\R)$
    \STATE $D:=\dim{\I}$
    \STATE $\tilde G:=G$ with $x_{n-D+i}$ replaced by
        $-\sum_{j=1}^{n-D}a_{ij}x_{j}$ for all $i>0$
    \STATE $\tilde\I:=\langle\tilde G\rangle_{\tilde\P}$
    \STATE $\H:=\textsc{PommaretBasis}\,(\tilde \I)$
    \STATE \textbf{return} $\deg{\H}-1$
  \end{algorithmic}
\end{algorithm}

The algorithm simply adds $D$ linear forms $y_{i}$ of the special form
(\ref{eq:mintrafo}).  The occuring coefficients $a_{ij}$ are then considered
as undetermined parameters.  Replacing in the ideal $\I$ every variable
$x_{n-D+i}$ with $i>0$ by $-\sum_{j=1}^{n-D}a_{ij}x_{j}$, we obtain a new
homogeneous ideal $\tilde\I$ in the polynomial ring
$\tilde\P=\kk(a_{ij})[x_{1},\dots,x_{n-D}]$ over the field of rational
functions in the $D(n-D)$ parameters $a_{ij}$ and compute its Pommaret basis
(see \cite{wms:comb1,wms:comb2} and references therein).

\begin{theorem}
  Algorithm \ref{alg:rednum} correctly determines $r(\R)$.
\end{theorem}

\begin{proof}
  We consider first the addition of $D$ generic linear forms
  $z_{i}=\sum_{j=1}^{n}b_{ij}x_{j}$ to the ideal $\I$.  This leads to an ideal
  $\hat\I$ in the polynomial ring $\hat\P=\kk(b_{ij})[x_{1},\dots,x_{n}]$
  depending on $Dn$ parameters and $n$ variables.  It follows from the
  classical proof of the existence of a Noether normalisation (see e.\,g.\
  \cite[Thm.~3.4.1]{gp:sing2}) over an infinite field that $\hat\I$ is a
  zero-dimensional ideal (which thus possesses a finite Pommaret basis).

  We now claim that the absolute reduction number $r(\R)$ is one less than the
  Castelnuovo-Mumford regularity $\reg{\hat\I}$.  According to
  \cite[Cor.~9.5]{wms:comb2}, $\reg{\hat\I}$ is given by the degree of the
  Pommaret basis of $\hat\I$, so that this claim implies that $r(\R)$ can be
  read off the Pommaret basis of $\hat\I$.  The correctness of the claim
  follows from a simple genericity argument.

  We build recursively $\kk(b_{ij})$-linear generating systems of the vector
  spaces $\hat\I_{q}$ for all degrees $q=1,2,\dots$ by taking all elements of
  $\H$ of degree $q$ and adding all products of elements of the previous
  generating system multiplied with a variable $x_{j}$.  We collect the
  coefficients of the obtained generators in a matrix.  Entering generic
  values for the parameters $b_{ij}$ leads to the maximal possible rank of
  this matrix and thus to the lowest possible dimension of the complement of
  the degree $q$ component of the corresponding specialisation of $\hat\I$.
  The absolute reduction number is the largest value of $q$ for which we
  cannot achieve a zero-dimensional complement.  Hence a generic choice of the
  parameters leads to the correct value of the absolute reduction number
  $r(\R)$.  Since computing over $\kk(b_{ij})$ corresponds to the generic
  branch of the parametric computation and since for a zero-dimensional ideal
  $\reg{\hat\I}$ is the lowest degree $q$ where $\hat\I_{q}=\hat\P_{q}$, we
  conclude that our claim is correct.

  Now consider the $(D\times n)$-matrix $(b_{ij})$: if the determinant of the
  submatrix composed of the last $D$ column does not vanish, then by a
  Gaussian elimination we obtain a set of linear forms $y_{i}$ in the
  ``reduced'' triangular form (\ref{eq:mintrafo}) leading to the same ideal
  $\hat\I$.  As the intersection of two Zariski open sets is again Zariski
  open, this observation proves that generically also the reduced ansatz
  (\ref{eq:mintrafo}) used in our algorithm yields the correct absolute
  reduction number.  Because of the special form of this ansatz, we may solve
  the linear forms for the variables $x_{n-D+i}$ and then perform the
  computations in the polynomial ring $\tilde\P$ depending only on $D(n-D)$
  parameters and $n-D$ variables. 
\end{proof}
 
\begin{remark}
  Since the Algorithms \ref{algo:wdstrafo} and \ref{alg:rednum} are based on
  Gr\"obner or Pommaret bases and the worst case complexity of computing
  Gr\"obner bases is doubly exponential in the number of variables (as shown
  by Mayr and Meyer \cite{mm:word}), we conclude that the complexity of these
  algorithms is also doubly exponential in the number of variables.
\end{remark}

\begin{example}\label{ex:weis}
  For $n=4$, the homogenised Weispfenning94 ideal
  $\I\lhd\kk[x_{1},\dots,x_{4}]$ is generated by the polynomials
$$  \begin{array}{ll}
    f_{1}&=x_{2}^4+x_{1}x_{2}^2x_{3}+x_{1}^2x_{4}^2-2x_{1}x_{2}x_{4}^2+
          x_{2}^2x_{4}^2+x_{3}^2x_{4}^2\,,\\
    f_{2}&=x_{1}x_{2}^4+x_{2}x_{3}^4-2x_{1}^2x_{2}x_{4}^2-3x_{4}^5\,,\\
    f_{3}&=-x_{1}^{3}x_{2}^2+x_{1}x_{2}x_{3}^3+x_{2}^4x_{4}+
          x_{1}x_{2}^2x_{3}x_{4}-2x_{1}x_{2}x_{4}^3\,.
  \end{array}$$
  Here $D=2$ and we replace $x_{4}$ by $-(a_{4,1}x_{1}+a_{4,2}x_{2})$ and
  $x_{3}$ by $-(a_{3,1}x_{1}+a_{3,2}x_{2})$ in $\I$ to obtain the new ideal
  $\tilde\I \lhd \kk(a_{3,1},a_{3,2},a_{4,1},a_{4,2}) [x_{1},x_{2}]$.  We
  compute a Pommaret basis $\H$ of $\tilde\I$ and get as leading terms
  \begin{displaymath}
    \lt{\H}=\bigl\{ x_1^4, x_1^3x_2^2, x_1^2x_2^3, x_1x_2^5, x_2^6 \bigr\}\,.
  \end{displaymath}
  Therefore $r(\R)=6-1=5$.
\end{example}
 
Our second example proves that there cannot exist a ``simple'' algorithm for
computing the (absolute) reduction number.  By ``simple'' we mean that the
algorithm uses exclusively information obtained from the leading terms (like
for instance Algorithm \ref{algo:wdstrafo} to transform into weakly $D$-stable
position). 

\begin{example} 
  We consider again Example \ref{ex:green} of Green. It follows immediately
  from the above presented bases that here $r(\R)=1<2=r(\P/\lt{\I})$.
  Following Algorithm \ref{alg:rednum}, we replace $x_3$ by $-(a_1x_1+a_2x_2)$
  in order to obtain the ideal $\tilde\I$.  Then we compute a Pommaret basis
  $\H$ of $\tilde\I$ and get for the leading terms
  \begin{displaymath}
    \lt{\H}=\bigl\{x_1^2, x_1x_2, x_2^2\bigr\}\,.
  \end{displaymath}
  Hence our algorithm yields the correct result $r(\R)=1$.  Since
  ${\mathcal{L}}=\lt{\I}$ is in fact even strongly stable, we conclude that
  $\gin{{\mathcal{L}}}={\mathcal{L}}$.  Hence the leading terms of the
  generators of $\I$ cannot contain any information on how to transform $\I$
  into a position such that the transformed ideal and its leading ideal share
  the same reduction number.
\end{example}

\section{Big Reduction Numbers and Gr\"obner Systems}
\label{sec:GS}

We present now an approach that is able to determine the whole reduction
number set $\mathrm{rSet}(\R)$ and thus in particular both the absolute and
the big reduction number.  Our method is based on the theory of Gr\"obner
systems, a notion introduced by Weispfenning \cite{vw:compgb} who also
provided a first algorithm for computing such systems.  Subsequently,
improvements and alternatives were presented by many authors
\cite{ksw:compgs,ksw:pps,am:gbp,mw:gbp,ss:cgb}.  Our calculations were done
using a \textsc{Maple} implementation of the \textsc{DisPGB} algorithm of
Montes which is available at \url{http://amirhashemi.iut.ac.ir/softwares}.

In the sequel, we denote by $\tilde\P=\P[\av]=\kk[\av,\xv]$ a
\emph{parametric} polynomial ring where $\av=a_1,\ldots,a_m$ represents the
parameters and $\xv=x_1,\ldots,x_n$ the variables.  Let $\prec_{\xv}$
(resp. $\prec_{\av}$) be a term order for the power products of the variables
$x_i$ (resp.\ the parameters $a_i$).  Then we introduce the block elimination
term order $\prec_{\xv,\av}$ in the usual manner: for all
$\alpha,\gamma\in\NN_{0}^{n}$ and all $\beta,\delta\in\NN_{0}^{m}$, we define
$\av^\delta\xv^\gamma\prec_{\xv,\av}\av^\beta\xv^\alpha$, if either
$\xv^\gamma \prec_{\xv} \xv^\alpha$ or $\xv^\gamma =\xv^\alpha$ and
$\av^\delta\prec_{\av} \av^\beta$.

\begin{definition}\label{def:GS}
  A finite set of triples $\bigl\{(\tilde G_i,N_i,W_i)\bigr\}_{i=1}^{\ell}$
  with finite sets $\tilde G_{i}\subset\tilde\P$ and
  $N_{i},W_{i}\subset\Q=\kk[\av]$ is a \emph{Gr\"obner system} for a
  parametric ideal $\tilde\I\unlhd\tilde\P$ with respect to the block order
  $\prec_{\xv,\av}$, if for every index $1\leq i\leq\ell$ and every
  specialisation homomorphism $\sigma:\Q\rightarrow\kk$ such that
  \begin{equation}\label{eq:spec}
    \text{\upshape (i) } \forall g\in N_i: \sigma(g)=0\,,\qquad
    \text{\upshape (ii) } \forall h\in W_i: \sigma(h)\neq0
  \end{equation}
  $\sigma(\tilde G_i)$ is a Gr\"obner basis of $\sigma(\tilde\I)\unlhd\P$ with
  respect to the order $\prec_{\xv}$ and if for any point $a\in\kk^{m}$ an
  index $1\leq i\leq\ell$ exists such that $a\in\V(N_{i})\setminus\V(W_{i})$.
\end{definition}

Thus a Gr\"obner systems yields a Gr\"obner basis for all possible values of
the parameters $\av$.  Weispfenning \cite[Theorem 2.7]{vw:compgb} proved that
every parametric ideal $\I\unlhd\S$ possesses a Gr\"obner system, but in
general the system is not unique.  Basically every algorithm (in particular
the \textsc{DisPGB} algorithm used by us) produces Gr\"obner systems such that
given one specific triple $(\tilde G_i,N_i,W_i)$ all specialisations $\sigma$
satisfying (\ref{eq:spec}) yield the same leading terms $\lt{\sigma(G_{i})}$
so that we can speak of a monomial ideal ${\mathcal{L}}_{i}\unlhd\P$
determined by the conditions $(N_i,W_i)$.  In the sequel, we will always
assume that a Gr\"obner system with this property is used.  As a simple
corollary, we find then that the reduction number set of an ideal $\I\unlhd\P$
is always finite.  Our proof also yields an explicit method for computing it.

\begin{theorem}\label{thm:redset}
  Let $\I\unlhd\P$ be a homogeneous ideal. Then its reduction number set
  $\mathrm{rSet}(\R)$ is finite.
\end{theorem}

\begin{proof}
  By definition, any minimal reduction of $\I$ is induced by $D$ linear forms
  \begin{equation}\label{eq:y}
    y_{i}=\sum_{j=1}^{n}a_{i,j}x_{j}\,,\qquad i=1,\dots,D
  \end{equation}
  with $a_{i,j}\in\kk$ and minimality is equivalent to $\J=\I+\langle
  y_{1},\dots,y_{D}\rangle$ being a zero-dimensional ideal.  Considering the
  coefficients $a_{i,j}$ as parameters, we may identify $\J$ with a parametric
  ideal $\tilde\I\unlhd\tilde\P$.  Let $\bigl\{(\tilde
  G_i,N_i,W_i)\bigr\}_{i=1}^{\ell}$ be a Gr\"obner system for $\tilde\I$.
  Without loss of generality, we may assume that for the first $s$ triples the
  associated monomial ideals ${\mathcal{L}}_{i}$ are zero-dimensional, whereas
  all other triples lead to monomial ideals of positive dimension.  Hence
  precisely the parameter values satisfying one of the conditions
  $(N_{i},W_{i})$ with $1\leq i\leq s$ define minimal reductions.  If $d_{i}$
  is the highest degree such that $({\mathcal{L}}_{i})_{d_{i}}\neq\P_{d_{i}}$,
  then it follows that $\mathrm{rSet}(\R)=\{d_{1},\dots,d_{s}\}$. 
\end{proof}

\begin{remark}\label{rem:gen}
  Any Gr\"obner system for a parametric ideal $\tilde\I$ contains one generic
  branch where the set $N_{i}$ of equations is empty.  Obviously, the
  corresponding leading ideal ${\mathcal{L}}_{i}$ must be the generic initial
  ideal $\gin{\I}$ and we have $d_{i}=r(\R)$.  This observation immediately
  yields an alternative proof of \cite[Cor.~2.2]{nvt:conrednum}: for almost
  all minimal reductions $\J$ we find $r_{\J}(\R)=r(\R)$.
\end{remark}

\begin{example}
  Let us consider again Green's Example \ref{ex:green} where $D=1$.  Hence we
  set $\tilde\I=\langle x_1^2, x_1x_3, x_2^2+x_1x_2,
  a_1x_1+a_2x_2+a_3x_3\rangle$.  The Gr\"obner system for $\tilde\I$ consists
  of $4$ triples.  For simplicity, we present in the following list for each
  branch as first entry only the corresponding leading ideal
  ${\mathcal{L}}_{i}$; the other two entries are the equations $N_{i}$ and the
  inequations $W_{i}$, respectively.
  \begin{displaymath}
    \begin{array}{llll}
      \{x_1,\ x_2^2,\ x_3^2,\ x_2x_3\}\qquad  &\{\} & \{a_1,\ a_2,\ a_1-a_2\}\\
      \{x_1,\ x_2^2,\ x_3^2,\ x_2x_3\} & \{a_1-a_2\}\qquad & \{a_2\}\\
      \{x_1,\ x_2^2,\ x_3^2\} & \{a_2\} & \{a_1\}\\
      \{x_2,\ x_1^2,\ x_3^2,\ x_1x_3\} & \{a_1\} & \{\}\\
    \end{array}
  \end{displaymath}
  We observe that all four branches lead to zero-dimensional leading ideals
  and their reduction numbers are $1,1,2,1$, respectively.  Therefore,
  $\mathrm{rSet}(\R)=\{1,2\}$ and $br(\R)=2$.
\end{example}

\begin{remark}
  For comparison, we briefly outline Trung's constructive characterisation
  \cite{nvt:conrednum} of the big reduction number of an ideal.  He also takes
  $D$ linear forms (\ref{eq:y}) with undetermined coefficients $a_{i,j}$ and
  proceeds with the ideal $\J=\I+\langle y_{1},\dots,y_{D}\rangle\unlhd\P$
  (note that he does not work in the parametric polynomial ring $\tilde\P$).
  Then he introduces the matrix $M_d$ of the coefficients of the generators in
  a $\kk$-linear basis of $\J_{d}$ (which are elements in $\Q$).  Let $\V_d$
  be the variety of the ideal generated in $\Q$ by all the minors of $M_d$ of
  the size of the number of terms of degree $d$.  Then, $br(\R)$ is the
  largest $d$ such that $\V_d\ne\V_{d+1}$ \cite[Cor.~2.3]{nvt:conrednum}.

  Note, however, that a priori it is unclear how to detect that one has
  obtained the largest $d$ with this property.  Thus his approach becomes
  truely algorithmic only by combining it with another result of his, namely
  that $br(\R)+1$ is bounded by the Castelnuovo-Mumford regularity
  $\mathrm{reg}(\I)$ \cite[Prop.~3.2]{nvt:redexp}.  Now one can check all
  degrees $d$ until $\mathrm{reg}(\I)$---which has to be computed first---and
  then finally decide on the value of $br(\R)$.  While the computation of a
  Gr\"obner system is surely a rather expensive operation, we strongly believe
  that it is much more efficient that the determination and subsequent
  analysis of large determinantal ideals.  Furthermore, our approach yields
  directly all possible values for the reduction number, whereas Trung must
  consider one determinantal ideal after the other (of increasing size).
\end{remark}

% Let $\I\unlhd\P$ be a homogeneous ideal in $D$-stable position and let
% $\alpha$ be the maximal degree of a minimal generator $t$ of $\lt{\I}$ lying
% in $\kk[x_1,\ldots,x_{n-D}]$.  Then \cite[Thm.~4.1]{ah:cmr} asserts that the
% Castelnuovo-Mumford regularity $\mathrm{reg}(\I)$ is $\alpha$.  Combining this
% with the above mentioned estimate $br(\R)<\mathrm{reg}(\I)$, we see that we
% may assume that $t=x_{n-D}^\alpha$ and furthermore that $br(\R)=\alpha-1$.

Finally, we note that Trung \cite{nvt:conrednum} proved that $br(\R)\le
br(\P/\lt{\I})$ if $\R$ is Cohen-Macaulay.  He also claimed that generally one
cannot compare $br(\R)$ and $br(\P/\lt{\I})$.  However, he did not provide a
concrete example where the above inequality is violated---which we will do
now.

\begin{example}
  Consider for $n=3$ the ideal
  \begin{displaymath}
    \I=\langle x_{1}^2x_{2}+x_{1}x_{2}^2,\
    x_{2}^3+x_{2}^2x_{3},\ x_{1}x_{3}^5, x_{2}^2x_{3}^5,\
    x_{1}^2x_{3}+x_{1}x_{2}x_{3},\ x_{1}^3-x_{1}x_{2}^2 \rangle\,.
  \end{displaymath}
  The given generators form already a Gr\"obner basis and thus $D=1$.
  $\lt{\I}$ is quasi-stable and, using Pommaret bases, one
  easily shows that the depth of $\R$ is $0$ and $\R$ is not Cohen-Macaulay.
  With $\J=\I+\langle x_{1}+x_{2}+x_{3} \rangle$, a simple computation yields
  that $\lt{\J}=\langle x_{1}, x_{2}x_{3}^2, x_{2}^2x_{3}, x_{2}^3, x_{3}^7
  \rangle$ and thus $br(\R)\ge r_{\J}(\R)=6$.  For showing that $br(\R)=6$, we
  set $\tilde\I=\I+\langle a_{1}x_{1}+a_{2}x_{2}+a_{3}x_{3}
  \rangle\unlhd\tilde\P$.  The Gr\"obner system of this ideal shows that the
  reduction numbers of the zero-dimensional branches are $3,5,6$,
  respectively, and therefore $br(\R)=6$.  On the other hand, $\lt{\I}=\langle
  x_{1}^2x_{3}, x_{2}^3, x_{1}^2x_{2}, x_{1}^3, x_{1}x_{3}^5, x_{2}^2x_{3}^5
  \rangle$.  We set $\tilde\I_{1}=\lt{\I}+\langle
  a_{1}x_{1}+a_{2}x_{s}+a_{3}x_{3} \rangle$, and compute its Gr\"obner system.
  Only three branches are zero-dimensional and they all have as reduction
  number $3$.  This shows that $br(\P/\lt{\I})=3<br(\R)$.
\end{example}

\section*{Acknowledgements.}  

The first author greatly appreciates financial support by DAAD (German
Academic Exchange Service) for a stay at Universit\"at Kassel in summer 2013.
He also would like to thank his coauthors for the invitation, hospitality, and
support.  The research of the first author was in part supported by a grant
from IPM (No. 92550420).

\bibliographystyle{splncs03}

\end{document}